\documentclass[reqno]{amsart}

\usepackage{amsfonts,color,newclude}
\usepackage{amsthm}
\usepackage{amssymb,transparent}
\usepackage{amsmath}
\usepackage{thm-restate}

\usepackage{mathabx}
\usepackage{graphicx}
\usepackage{subfiles}
\usepackage{tikz}
\usepackage{tikz-cd}
\usepackage[colorlinks=true, pdfstartview=FitV, linkcolor=blue, citecolor=blue, urlcolor=blue]{hyperref}
\usepackage{hyperref}

\usepackage[nameinlink]{cleveref}

\renewcommand{\tilde}{\widetilde}
\newcommand{\mathsym}[1]{{}}

\newcommand{\naturals}{\mathbb{N}}

\newcommand{\diam}{\operatorname{diam}}

\newcommand{\co}{\colon\thinspace}

\newcommand{\oskel}{^{(1)}}

\renewcommand{\bar}{\overline}

\newcommand{\Stab}{\operatorname{Stab }}

\newcommand{\cal}[1]{\mathcal{#1}}

\def\mc {\mathcal}

\newcommand{\hyp}{\mathfrak{h}}		

\newcommand{\leftQ}[2]{\left.\raisebox{-.2em}{$#2$}\middle\backslash\raisebox{.2em}{$#1$}\right.}
\newcounter{probnum}
\newtheorem{theorem}{Theorem}[section]
\newtheorem{definition}[theorem]{Definition}

\newtheorem{proposition}[theorem]{Proposition}
\newtheorem{cor}[theorem]{Corollary}
\newtheorem{lemma}[theorem]{Lemma}
\newtheorem{example}[theorem]{Example}

\newtheorem{remark}[theorem]{Remark}

\newtheorem{hypotheses}[theorem]{Hypotheses}

\newtheorem{obs}[theorem]{Observation}

\title[Separation and Relative Quasi-convexity Criteria for Relatively Geometric Actions]{Separation and Relative Quasi-convexity Criteria for Relatively Geometric Actions}

\author{Eduard Einstein}
\author{Daniel Groves}
\author{Thomas Ng}

\begin{document}
	\maketitle
	
	\setcounter{tocdepth}{1}

\begin{abstract}
Bowditch characterized relative hyperbolicity in terms of group actions on fine hyperbolic graphs with finitely many edge orbits and finite edge stabilizers. In this paper, we define generalized fine actions on hyperbolic graphs, in which the peripheral subgroups are allowed to stabilize finite sub-graphs rather than stabilizing a point.  Generalized fine actions are useful for studying groups that act relatively geometrically on a CAT(0) cube complex, which were recently defined by the first two authors.  Specifically, we show that a group acting relatively geometrically on a CAT(0) cube complex admits a generalized fine action on the one-skeleton of the cube complex. For generalized fine actions, we prove a criterion for relative quasiconvexity as subgroups that cocompactly stabilize quasi-convex sub-graphs, generalizing a result of Martinez-Pedroza and Wise in the setting of fine hyperbolic graphs. As an application, we obtain a characterization of boundary separation in generalized fine graphs and use it to prove that Bowditch boundary points in relatively geometric actions are always separated by a hyperplane stabilizer.
\end{abstract}


\section{Introduction}

There are many equivalent formulations of relatively hyperbolic groups, see for example \cite{BowditchRH,DrutuSapir,FarbRelHypGroups,Gromov87,GM08,Hruska2010,OsinRelHyp}.	In \cite{BowditchRH}, Bowditch describes relative hyperbolicity in terms of an action on a fine hyperbolic graph with certain finiteness conditions (see Definition~\ref{def:fine} for the definition of `fine').
	A natural example of such a graph is the coned-off Cayley graph of a relatively hyperbolic pair, defined by Farb \cite{FarbRelHypGroups}. Unfortunately, fineness and other important finiteness properties of the action of $(G,\mc{P})$ on the coned-off Cayley graph are not  quasi-isometry invariants. In \cite{RelCannon}, the first and second author introduced the notion of a relatively hyperbolic pair $(G,\mc{P})$ acting \emph{relatively geometrically} on a CAT(0) cube complex $\tilde{X}$.  
In this situation, a result of Charney and Crisp, \cite[Theorem 5.1]{CharneyCrisp}, implies that $\tilde{X}$ and its $1$--skeleton $\tilde{X}\oskel$ is quasi-isometric to the coned-off Cayley graph of $(G,\mc{P})$. However, the edge stabilizers of $\tilde{X}\oskel$ are often infinite.
	
	In this paper, we develop tools to translate geometric features of a generalized fine action on a hyperbolic graph to the Bowditch boundary of $(G,\mc{P})$.
We also apply these tools to prove some fundamental results about relatively hyperbolic groups that act relatively geometrically on CAT(0) cube complexes.
	The Bowditch boundary, a compact boundary for a relatively hyperbolic pair $(G,\mc{P})$, was first introduced by Bowditch in \cite{BowditchRH}. 
	One way Bowditch constructs this boundary is from a fine hyperbolic graph $\Gamma$ that witnesses the relative hyperbolicity of $(G,\mc{P})$. 
	As a set, the Bowditch boundary is the disjoint union of the visual boundary of $\Gamma$ with the vertices of $\Gamma$ that have infinite stabilizer. Bowditch endows this set with the topology we describe in Definition~\ref{D: fine boundary}. 
	In our study of relatively geometric actions, we would like to take advantage of the correspondence between the CAT(0) cube complex $\tilde{X}$ and the coned-off Cayley graph for $(G,\mc{P})$ to use the geometry of $\tilde{X}$ to prove statements about the Bowditch boundary of $(G,\mc{P})$.

	Let $(G,\mc{P})$ be a relatively hyperbolic group pair where $G$ acts by isometries on a CAT(0) cube complex $\tilde{X}$. The action $G$ is \textbf{relatively geometric} if:
\begin{enumerate}
\item the quotient of $\leftQ{\tilde{X}}{G}$ is compact,
\item every infinite cell stabilizer is a finite index subgroup of $P^g$ for some $P\in \mc{P}$ and $g \in G$,
\item every $P\in\mc{P}$ acts elliptically.
\end{enumerate}

If $(G,\mc{P})$ acts relatively geometrically on $\tilde{X}$ then $\tilde{X}$ (and also its $1$--skeleton $\tilde{X}^{(1)}$) are quasi-isometric to the coned-off Cayley graph of $(G,\mc{P})$, by \cite[Theorem 5.1]{CharneyCrisp}.  However, in general the graph $\tilde{X}^{(1)}$ is not a fine graph.  Thus,
to help study actions similar to relatively geometric actions, we introduce generalized fine actions on a hyperbolic graph:
\begin{definition}\label{D: genfine}
Let $G$ be a group that acts by isometries on a $\delta$--hyperbolic graph $\Gamma$ and let $\mc{P}$ be a finite and almost malnormal collection of finitely generated subgroups of $G$.
For each $P\in \mc{P}$ and $g\in G$, let $\Gamma_{P^g}$ be the subgraph of $\Gamma$ whose cells have stabilizer commensurable to $P^g$.
We say that $\Gamma$ is \textbf{generalized fine with respect to the action of $(G,\mc{P})$} if:
\begin{enumerate}
\item the quotient $\leftQ{\Gamma}G$ is compact,
\item every cell with infinite stabilizer lies in $\Gamma_{P^g}$ for some $P\in\mc{P}$ and $g\in G$,
\item each subgraph $\Gamma_{P^g}$ is compact and connected, and
\item for every $n\in\naturals$, every edge with finite stabilizer lies in finitely many circuits of length $n$.
\end{enumerate}
\end{definition}
It is immediate from the definitions that a fine graph is also generalized fine. Generalized fine actions witness relative hyperbolicity. We prove in Proposition~\ref{P: gen fine rel hyp} that $(G,\mc{P})$ in Definition~\ref{D: genfine} is a relatively hyperbolic pair. 
Relatively geometric actions immediately give rise to generalized fine actions: in Example~\ref{E: relgeom is genfine} below we show that $\tilde{X}\oskel$ is generalized fine with respect to the induced action of $(G,\mc{P})$.

	We say that $H\le K$ has a \textbf{(quasi-)convex cocompact core} if $K$ stabilizes a (quasi-)convex $H$--cocompact connected subgraph (see Definition~\ref{D: core} for a precise definition). 
Our first main result shows that the existence of a quasi-convex cocompact core that interacts nicely with the sub-graphs stabilized by peripheral conjugates implies relative quasi-convexity:
	\begin{restatable}{theorem}{qccriterion}\label{Thm: rqc criterion}  
	Let $(G,\mc{P})$ be a relatively hyperbolic pair and let $\Gamma$ be a hyperbolic graph with a $G$--action so that $\Gamma$ is generalized fine with respect to the action of $(G,\mc{P})$. For any $P\in \mc{P}$ and $g\in G$, let $\Gamma_{P^g}$ be the sub-graph of $\Gamma$ whose cell stabilizer are commensurable to $P^g$. 
	If $H\le G$ has a quasi-convex cocompact core $\Gamma_H$ and one of the following hold:
	\begin{itemize}
	\item $\Gamma$ is fine or
	\item for all $P\in \mc{P}$ and $g\in G$, $\Gamma_H\cap \Gamma_{P^g} \ne\emptyset$ implies $|P^g\cap H|=\infty$,
	\end{itemize}
	then $H$ is relatively quasi-convex in $(G,\mc{P})$. 
	\end{restatable}
	 
As with the definition of relative hyperbolicity, there are many equivalent characterizations of relative quasi-convexity, see \cite{Hruska2010}. As a consequence of Theorem~\ref{Thm: rqc criterion} we provide an alternate proof of a theorem of Martinez-Pedroza and Wise \cite{WiseMP10} characterizing relative quasi-convexity in terms of quasi-convex cores in fine hyperbolic graphs:
	\begin{restatable}{cor}{finecriterion} $($\cite[Theorem 1.7]{WiseMP10}$)$\label{C: fine hyperbolic rqc criterion}
	Let $(G,\mc{P})$ be a relatively hyperbolic pair acting cocompactly on a fine hyperbolic graph so that every edge stabilizer is finite. A subgroup $H\le G$ is relatively quasi-convex in $(G,\mc{P})$ if and only if $H$ has a quasi-convex cocompact core in $\Gamma$. 
	\end{restatable}
As an application, we see that if $(G,\mc{P})$ acts relatively geometrically on a CAT(0) cube complex $\tilde{X}$, the stabilizer of a hyperplane in $\tilde{X}$ is relatively quasi-convex in $(G,\mc{P})$. 
The first and third author also will use Corollary~\ref{C: fine hyperbolic rqc criterion} to construct relatively geometric actions of $C'(\frac16)$--small cancellation free products with relatively geometrically cubulated factor groups.

Following the criterion in \cite{BergeronWise} for proper and cocompact cubulations of hyperbolic groups, the first two authors show in \cite[Theorem 2.6]{RelCannon} that if any two points in the Bowditch boundary $\partial_{\mc{D}}K$ of a relatively hyperbolic pair $(K,\mc{D})$ lie in $H$--distinct components of the limit set of a full relatively quasi-convex $H\le K$, then $K$ acts relatively geometrically on a CAT(0) cube complex. 
Our other main result shows that any pair of distinct points in the Bowditch boundary of a relatively geometrically cubulated group can be separated by the limit set of a hyperplane stabilizer:
\begin{restatable}{theorem}{relgeomseparation}\label{T: hyperplane separation theorem}	Let $(K,\mc{D})$ act relatively geometrically on a CAT(0) cube complex $\tilde{X}$. If $x,y\in\partial_{\mc{D}}K$ and $x\ne y$, then there exists a hyperplane $W$ of $\tilde{X}$ and a finite index subgroup $K_W \le \Stab_K(W)$ so that $x,y$ are in $K_W$--distinct components of $\partial_{\mc{D}}K\setminus \Lambda K_W$.
\end{restatable}
More generally, if a graph $\Sigma$ is generalized fine with respect to the action of a relatively hyperbolic pair $(K,\mc{D})$, we obtain a technical criterion for deciding when a subgroup of $K$ with a cocompact core in $\Sigma$ separates two points in the Bowditch boundary, see Theorem~\ref{T: separation criterion}. 
Both Theorem~\ref{T: hyperplane separation theorem} and Theorem~\ref{T: separation criterion} are essential tools in the first and third authors' forthcoming work on relative cubulations for small cancellation free products.

The idea behind Theorem~\ref{T: hyperplane separation theorem} is to take a hyperplane $W$ that is dual to an edge of a combinatorial geodesic between $x$ and $y$. 
However, the Bowditch boundary is not the visual boundary of $\tilde{X}$. Moreover, $\tilde{X}\oskel$ is not proper, so it is not clear that such a combinatorial geodesic exists. 
In order to make statements about $\partial_{\mc{D}}K$, we need to \emph{coarsely} translate the geometric features of $W$ and $\tilde{X}$ to a fine hyperbolic graph $\Sigma'$ with a $K$-action that witnesses the relative hyperbolicity of $(K,\mc{D})$. 
While the image of $W$ in $\Sigma'$ will separate $\Sigma'$ into two components, we need to ensure that the limit set of $\Stab_G (W)$ actually separates $x$ and $y$ into distinct complementary components of $\partial_{\mc{D}}K$, which is still not the visual boundary $\Sigma'$.

\subsection{Outline: }
We introduce some background on relatively hyperbolic groups and fine hyperbolic graphs in Section~\ref{S: background}. We also discuss a construction similar to that in \cite{FarbRelHypGroups} to relate paths in two graphs $\Sigma,\Sigma'$ where $\Sigma'$ is formed by collapsing some of the edges of $\Sigma$. Then, we recall some specific properties of relatively geometric actions in Section~\ref{S: relgeom background}. 
In Section~\ref{S: genfine}, we explore the properties of generalized fine actions and show that relatively geometric actions on a cube complex give rise to generalized fine actions on the one-skeleton of the cube complex. 
We then prove Theorem~\ref{Thm: rqc criterion} and Corollary~\ref{C: fine hyperbolic rqc criterion} in Section~\ref{S: rqc technical}. 

The main result of Section~\ref{S: separation criterion} is Theorem~\ref{T: separation criterion}, a technical separation criterion for points in the Bowditch boundary of a relatively hyperbolic group in terms of a generalized fine action. 
Finally in Section~\ref{S: relgeom separation}, we prove Theorem~\ref{T: hyperplane separation theorem} using Theorem~\ref{T: separation criterion}. 

\subsection*{Acknowledgements}
The first author would like to thank Jason Manning for useful conversations that helped shape this work.
EE was partially supported by an AMS--Simons travel grant and a postdoctoral fellowship from the Swartz Foundation.
DG was partially supported by NSF Grants DMS--1904913 and DMS--2037569.
TN was partially supported by ISF grant 660/20, an AMS--Simons travel grant, and at the Technion by a Zuckerman Fellowship.

\section{Obtaining Fine Hyperbolic Graphs from Relatively Geometric Actions}\label{S: background}
	
	\subsection{Fine hyperbolic graphs and relative hyperbolicity.}
	
First, we recall the definition of a fine graph:
\begin{definition}\label{def:fine}
Let $\Gamma$ be a graph. A \textbf{circuit} is an embedded loop in $\Gamma$. The graph $\Gamma$ is \textbf{fine} if for each edge $e$ of $\Gamma$ and every $n\in\naturals$, there exist finitely many circuits of length $n$ containing $e$.
\end{definition}

Here is Bowditch's definition of relative hyperbolicity in terms of fine hyperbolic graphs:
 
\begin{definition}[{\cite[Definition 2]{BowditchRH}, written as stated in \cite[Definition 3.4 (RH-4)]{Hruska2010}}]\label{D: fine relhyp}
Suppose $G$ acts on a $\delta$--hyperbolic graph $\Gamma$ with finite edge stabilizers and finitely many $G$--orbits of edges. If $K$ is fine, and $\mc{P}$ is a set of representatives of the conjugacy classes of infinite vertex stabilizers, then $(G,\mc{P})$ is a \textbf{relatively hyperbolic pair}.
\end{definition}

In Section~\ref{S: rqc technical}, we will also use a dynamical characterization of relative hyperbolicity and relative quasi-convexity due to Yaman~\cite{Yaman}.

As noted by Bowditch \cite[pg 3]{BowditchRH}, fineness is not a quasi-isometry invariant. Here is an example of quasi-isometric graphs where one graph is fine and the other is not:
\begin{example}
Let $\Gamma$ be a graph with $2$ vertices and a single edge joining the two vertices. Let $\Sigma$ be a graph with $2$ vertices and an infinite number of edges between the two vertices. Then $\Gamma$ and $\Sigma$ are quasi-isometric. 
Any edge of $\Sigma$ lies in infinitely many circuits of length $2$, so $\Sigma$ is not fine. 
\end{example}
	
\subsection{Electrification and De-Electrification}

Farb first introduced the notion of electrifying a space in \cite{FarbRelHypGroups} where electrification of the fundamental group of a finite volume hyperbolic 3-manifold is accomplished by collapsing the cosets of cusp subgroups to points. 
This idea inspires the definitions in this subsection. Similar constructions have also been performed in \cite{AbbottManning,BowditchRH,DahmaniMj,Spriano17}.

Let $G$ be a group and suppose $G$ acts by isometries on a graph $\Sigma$. Let $\mc{B}$ be a collection of pairwise disjoint and connected sub-graphs of $\Sigma$.
\begin{definition}\label{D: complete electrification}
The \textbf{complete electrification of $\Sigma$ (with respect to $\mc{B}$)} is the graph $\Sigma'$ that is constructed by contracting each $B\in\mc{B}$ to a vertex $v_B$ of $\Sigma'$.  There is a canonical quotient map $\sigma \co \Sigma \to \Sigma'$, which we call the \textbf{electrification} map.

The \textbf{stable part of $\Sigma$ under the electrification with respect to $\mc{B}$} is denoted 
\[\Sigma_0 = \Sigma\,\,\setminus\,\, \left(\bigcup_{B\in \mc{B}} B\right).\] 
\end{definition}

The stable part $\Sigma_0$ embeds naturally into both $\Sigma$ and $\Sigma'$. When $\mc{B}$ is clear from context, we refer to $\Sigma_0$ as the \textbf{stable part of $\Sigma$}. 
We are interested in how paths behave under electrification.

\begin{definition}\label{D: wo peripheral backtracking}
Let $\gamma$ be a path in $\Sigma$. The path $\gamma$ is \textbf{without peripheral backtracking} if for every $B\in\mc{B}$, $\gamma\,\setminus\, (\gamma\cap B)$ is connected when $\gamma$ is a loop and has at most $2$ components otherwise. 

Similarly, a path $\gamma'$ in $\Sigma'$ is \textbf{without peripheral backtracking} if for all $B\in\mc{B}$, $\gamma'\,\setminus\, v_B$ is connected when $\gamma$ is a loop and has at most $2$ components otherwise. 
\end{definition}

We can relate paths without peripheral backtracking in $\Sigma$ and $\Sigma'$ as follows:

\begin{definition}\label{D: path electrification}
Let $\gamma$ be a path in $\tilde{X}^{(1)}$ without peripheral backtracking. 
The \textbf{electrification} $\gamma'$ of $\gamma$ is the path in $\Sigma'$ constructed by collapsing sub-segments of the form $\gamma\cap B$ to $v_B$.

Similarly, let $\rho'$ be a path in $\Sigma'$ without peripheral backtracking. Let $\bar{\rho}'$ be the closure $\rho'\cap \Sigma_0$ in $\Sigma$. If $\bar{\rho}'\cap B$ fails to be connected, $\bar{\rho}'\cap B$ is exactly two points because $\rho'$ is without peripheral backtracking. A \textbf{complete de-electrification} of $\rho'$ is a path $\rho$ in $\Sigma$ constructed by joining any disconnected $\bar{\rho}'\cap B$ by an embedded path in $B$.
\end{definition}

\subsection{Obtaining an action on a fine hyperbolic graph from a relatively geometric action}
\label{S: relgeom background}

For the following subsection, assume that the relatively hyperbolic pair $(G,\mc{P})$ acts relatively geometrically on a CAT(0) cube complex $\tilde{X}$. 

The main goal of this section is to show that subgroups with convex cores in $\tilde{X}$ have quasi-convex cocompact cores in a fine hyperbolic graph that witnesses the relative hyperbolicity of $(G,\mc{P})$. 
\begin{proposition}\label{P: qc invt subraph}
Let $H\le G$ be a subgroup that stabilizes a convex sub-complex $\tilde{Y}\subseteq\tilde{X}$. Then $(G,\mc{P})$ acts on a fine hyperbolic graph $\Gamma$ so that:
\begin{enumerate}
\item every $P\in\mc{P}$ fixes a vertex of $\Gamma$,
\item every edge stabilizer is finite, 
\item $\leftQ{\Gamma}{G}$ is compact, and
\item there exists a quasi-convex $H$--invariant connected sub-graph $\Gamma_H\subseteq \Gamma$. 
\end{enumerate}
Additionally, if $\tilde{Y}$ is $H$--cocompact, $\Gamma_H$ is $H$--cocompact.
\end{proposition}

For each $P\in\cal{P}$ and $g\in G$, let $\tilde{X}_{P^g}$ denote the sub-graph of $\tilde{X}^{(1)}$ consisting of cells whose stabilizer is commensurable to $P^g$. Since infinite stabilizers in $\tilde{X}$ are commensurable with a unique $P^g$, the following is immediate from the definitions.

\begin{lemma}
If $P_1^{g_1} \ne P_2^{g_2}$ then $\tilde{X}_{P_1^{g_1}} \cap \tilde{X}_{P_2^{g_2}} = \emptyset$.
\end{lemma}

Before proving Proposition~\ref{P: qc invt subraph}, we investigate some of the properties of $\tilde{X}\oskel$ and its complete electrification with respect to the collection of sub-graphs of the form $\tilde{X}_{P^g}$. 

By \cite[Proposition 3.5]{RelGeom}, the sub-graph $\tilde{X}_{P^g}$ is the $1$--skeleton of a compact and convex sub-complex of a CAT(0) cube complex.
Then we obtain the following fact about $\tilde{X}_{P^g}$:

\begin{proposition}\label{stab subgraph nice}
The sub-graph $\tilde{X}_{P^g}$ is connected, convex and compact. 
\end{proposition}





\begin{proposition}\label{P: electric qi}
Let $\Gamma$ be the complete electrification of $\tilde{X}^{(1)}$ with respect to 
\[\mc{B} = \{\tilde{X}_{P^g}:\, P\in\mc{P},\,g\in G\}\] and let $\Gamma_0$ be the stable part of $\tilde{X}^{(1)}$. There exist a cocompact action of $G$ on $\Gamma$ and a $G$--equivariant quasi-isometry $f:\Gamma\to \tilde{X}^{(1)}$ so that $f|_{\Gamma_0}$ is the identity map. 
\end{proposition}

\begin{proof}
Observe that in $\tilde{X}^{(1)}$, $G\Gamma_0 \subseteq \Gamma_0$ because cells of $\tilde{X}^{(1)}$ in $\Gamma_0$ are precisely those with finite stabilizer. 
For each $g,h\in G$ and $P\in \mc{P}$, set $g\cdot v_{P^{h}} = v_{P^{gh}}$. 

The quotient $\leftQ{\Gamma}G$ differs from the quotient $\leftQ{\tilde{X}^{(1)}}G$ by collapsing finitely many edges of a finite graph, so $\leftQ{\Gamma}G$ is still compact. 

For each $\tilde{X}_{P^g}$ fix $x_{P^g}\in\tilde{X}_{P^g}$. 
We define $f:\Gamma\to \tilde{X}^{(1)}$ as follows:
\[f(x) = \begin{cases}
x & x\in\Gamma_0 \\ x_{P^g} & x = v_{P^g}
\end{cases}\]

There are finitely many $G$--orbits of $\tilde{X}_{P^g}$ and each $\tilde{X}_{P^g}$ is compact, so there exists $s\ge 0$ that uniformly bounds the diameter of all $\tilde{X}_{P^g}$. 
It follows immediately that $f$ is coarsely surjective.

Let $x,y$ be vertices of $\Gamma$ and choose a geodesic path $\gamma'$ connecting them in $\Gamma$. 
Let $D$ be the length of $\gamma'$ in $\Gamma$. 
There exists a de-electrification $\gamma$ of $\gamma'$ so that $\gamma'$ has length at most $D+Ds$, since $\gamma'$ encounters at most $D$ vertices of the form $v_{P^g}$ in its interior. 
Extend $\gamma'$ to a path $\gamma''$ joining $f(x)$ to $f(y)$ by adding segments of length at most $s$ to each end. 
The distance in $\tilde{X}^{(1)}$ between $f(x)$ and $f(y)$ is at most $D+Ds +2s$. 

We see that $f$ is distance non-decreasing because $\Gamma$ is formed from $\tilde{X}^{(1)}$ by collapsing sub-graphs to points. 
Hence $f$ is a quasi-isometry. 
\end{proof}

There is also a $G$--equivariant coarse inverse:
\begin{cor}\label{C: coarse inverse qi}
Let $c:\tilde{X}^{(1)}\to\Gamma$ be the map that collapses $X_{P^g}$ to $v_{P^g}$ and fixes $f(\Gamma_0)$. Then $c$ is a $G$--equivariant quasi-isometry that fixes the image of the stable part of $\Gamma_0$ in $\tilde{X}^{(1)}$. 
\end{cor}

\begin{proposition}\label{P: electric is fine}
The graph $\Gamma$ in Proposition~\ref{P: electric qi} is a fine hyperbolic graph. 
\end{proposition}

\begin{proof}
By \cite[Theorem 5.1]{CharneyCrisp}, $\tilde{X}^{(1)}$ is quasi-isometric to the coned-off Cayley graph for $(G,\mc{P})$, so $\tilde{X}^{(1)}$ is a hyperbolic graph.
By Proposition~\ref{P: electric qi}, $\Gamma$ is hyperbolic. 

Vertex stabilizers in $\Gamma$ are maximal parabolic and each maximal parabolic stabilizes exactly one point in $\Gamma$, so $\Gamma$ has finite pair stabilizers. 
By cocompactness, $\leftQ{\Gamma}G$ is finite, so by \cite[Lemma 4.5]{BowditchRH}, $\Gamma$ is a fine graph. 
\end{proof}

We are ready to prove \Cref{P: qc invt subraph}.
\begin{proof}[Proof of Proposition~\ref{P: qc invt subraph}]
By Proposition~\ref{P: electric qi}, the action of $(G,\mc{P})$ provides a cocompact action of $G$ on $\Gamma$ where each $P^g$ fixes $v_{P^g}$. 
Since $H$ stabilizes a convex sub-complex $\tilde{Y}\subseteq \tilde{X}$, $\tilde{X}^{(1)}\cap \tilde{Y} = \tilde{Y}^{(1)}$ is a convex sub-graph of $\tilde{X}^{(1)}$. When $\tilde{Y}$ is $H$--cocompact, $\tilde{Y}^{(1)}$ is also $H$--cocompact. 
The collapse $c:\tilde{X}^{(1)}\to \Gamma$ takes $\tilde{Y}^{(1)}$ to a sub-graph of $\Gamma_H$ of $\Gamma$. 
Since $c$ is a quasi-isometry by \Cref{C: coarse inverse qi}, $\Gamma_H$ is quasi-convex in $\Gamma$. 
\end{proof}

\section{Generalized Fine Actions}
\label{S: genfine}

The behavior in Section~\ref{S: relgeom background} is more general than relatively geometric actions on CAT(0) cube complex and is captured by the definition of generalized fine graphs (Definition~\ref{D: genfine}). 

\begin{hypotheses}\label{H: genfine}
Let $(K,\mc{D})$ be a relatively hyperbolic pair and let $K$ act on a graph $\Sigma$. For each $D\in\mc{D}$ and $k\in K$, let $\Sigma_{D^k}$ be the sub-graph of $\Sigma$ consisting of cells whose stabilizer is commensurable to $D^k$.
\end{hypotheses}

The main goal of this section is to characterize generalized fine actions as follows: 
\begin{enumerate}
\item If electrifying $\Sigma$ with respect to the $\Sigma_{D^k}$ results in a fine hyperbolic graph with an appropriate action, the action of $K$ on $\Sigma$ is generalized fine (Proposition~\ref{fine vs gen fine}). 
\item If the action of $(K,\mc{D})$ on $\Sigma$ is generalized fine, then electrifying the $\Sigma_{D^k}$ results in a fine hyperbolic graph (Proposition~\ref{P: gen fine rel hyp}).
\end{enumerate}

\begin{proposition}\label{fine vs gen fine}
Assume Hypotheses~\ref{H: genfine}. Suppose that
\begin{enumerate}
\item $K$ acts cocompactly on $\Sigma$,
\item every maximal parabolic $D\in\mc{D}$ stabilizes a vertex of $\Sigma$, and
\item the sub-graph $\Sigma_{D^k}$ is connected and compact.
\end{enumerate}
If the complete electrification with respect to $\{\Sigma_{D^k}:\,D\in \mc{D},\,k\in K\}$ is a fine hyperbolic graph, then $\Sigma$ is generalized fine with respect to the action of $(K,\mc{D})$. 
\end{proposition}

\begin{proof}
Fix $n\ge 0$ and let $e$ be an edge of $\Sigma$ with finite stabilizer.  
Let $\gamma$ be a length $n$--circuit in $\Sigma$ without peripheral backtracking. 
Let $\Gamma$ be the complete electrification of $\Sigma$. 
Since each of the $\Sigma_{D^k}$ is compact and there are finitely many $K$--orbits of $\Sigma_{D^k}$, there is a constant $L(n)$ that bounds the number of length at most $n$ embedded paths in $\Sigma_{D^k}$ from above. 

The map $\sigma:\Sigma\to \Gamma$ that collapses the sub-graphs $\Sigma_{D^k}$ to points $v_{D^k}$ is injective on $e$ because $e$ has finite stabilizer.
The electrification $\gamma'$ of $\gamma$ is a circuit in $\Gamma$ containing $\sigma(e)$. 
Since $\Gamma$ is a fine graph, there exists some $T\ge 0$ so that $\gamma'$ is one of $T$ circuits passing through $\sigma(e)$. 
Since $\gamma$ intersects at most $n$ of the $v_{P^g}$, $\gamma$ can be obtained as a de-electrification of $\gamma'$ where each of the $v_{P^g}$ that $\gamma$ intersects is replaced by an embedded path in the corresponding $\Sigma_{D^k}$.
Hence there are at most $T(L(n))^n$ possibilities for $\gamma$.
Thus every edge with finite stabilizer in $\Sigma$ is contained in only finitely many circuits. 
\end{proof}

Proposition~\ref{P: electric is fine} and Proposition~\ref{fine vs gen fine}  imply that a relatively geometric action gives rise to a generalized fine action. 
\begin{example}\label{E: relgeom is genfine}
If $(G,\mc{P})$ acts relatively geometrically on a CAT(0) cube complex $\tilde{X}$, $\tilde{X}^{(1)}$ is generalized fine with respect to the action of $(G,\mc{P})$. 
\end{example}

The proof of Proposition~\ref{P: electric qi} does not use the fact that $\tilde{X}$ is a CAT(0) cube complex, so we obtain the following natural analogue of Corollary~\ref{C: coarse inverse qi}:
\begin{proposition}\label{P: electric qi generalized}
Assume Hypotheses~\ref{H: genfine}. If $\Gamma$ is the complete electrification of $\Sigma$ with respect to $\{\Sigma_{D^k}:\, D\in\mc{D},\,k\in K\}$ then the electrification map $c:\Sigma\to \Gamma$ is a $K$--equivariant quasi-isometry and $c$ restricts to the identity on the stable part of $\Gamma$ embedded in $\Sigma$.
\end{proposition}

Generalized fine actions witness relative hyperbolicity:
\begin{proposition}\label{P: gen fine rel hyp}
Assume Hypotheses~\ref{H: genfine}. If $\Sigma$ is generalized fine with respect to the action of $(K,\mc{D})$, then the complete electrification $\Sigma'$ with respect to $\mc{B} = \{\Sigma_{D^k}:\,D\in\mc{D},k\in K\}$ is a fine (hyperbolic) graph. Therefore, $(K,\mc{D})$ is a relatively hyperbolic pair. 
\end{proposition}

\begin{proof}
Let $e'$ be an edge in $\Sigma'$ with finite stabilizer and let $\gamma'$ be a circuit of length $n$ in $\Sigma'$. Since $\gamma'$ is an embedded loop, $\gamma'$ is without peripheral backtracking. Let $\Sigma_0$ be the stable part of $\Sigma$. 
Let $e$ be the unique edge of $\Sigma$ whose interior is $e'\cap\Sigma_0$. 
The complete electrification of any de-electrification of $\gamma'$ returns $\gamma'$.

Therefore, there exists a circuit $\gamma$ containing $e$ so that $\gamma'$ is a complete electrification of $\gamma$. 
Since there are finitely many $K$--orbits of $\Sigma_{D^k}$, there is a uniform bound $L\ge 0$ on the diameters of the $\Sigma_{D^k}\in\mc{B}$. 
Hence the length of $\gamma$ is at most $n + nL$.
Since $\Sigma$ is generalized fine, and $\gamma$ contains $e$, an edge with finite stabilizer, there are only finitely many possibilities for $\gamma$.  Thus there are only finitely many possibilities for $\gamma'$. 

There is a natural action of $K$ on $\Sigma'$ defined as follows: let $x\in \Sigma'$ and $k\in K$.
When $x\in\Sigma_0$, then $k\cdot x$ is defined according to the action of $K$ on $\Sigma_0\subseteq \Sigma$.
Otherwise, $x\ = v_{D^{k_0}}$ for some $D\in\mc{D}$ and some $k_0\in K$, so define $k\cdot x = v_{D^{kk_0}}$.
There are finitely many $K$--orbits of edges because $K$ acts cocompactly on $\Sigma'$, and every edge has finite stabilizer because electrification collapses every edge with infinite stabilizer. 
By Definition~\ref{D: fine relhyp}, the hyperbolicity and fineness of $\Sigma'$ implies $K$ is hyperbolic relative to any finite set of conjugacy class representatives of the infinite vertex stabilizers.  The set $\mc{D}$ is such a finite set.
\end{proof}

\section{Quasi-convex sub-graphs of generalized fine graphs}
\label{S: rqc technical}

In this section, we prove Theorem~\ref{Thm: rqc criterion}. We first formally state the definition of a quasi-convex cocompact core:

\begin{definition}\label{D: core}
Let $G$ act on a hyperbolic graph $\Gamma$ by isometries and let $H\le G$ be a subgroup of $G$.
A \textbf{(quasi-)convex core of $H$ in $\Gamma$} is  a connected sub-graph $\Gamma_H$ so that:
\begin{enumerate}
\item the quotient $\leftQ{\Gamma_H}{H}$ is compact, and
\item $\Gamma_H$ is (quasi-)convex in $\Gamma$. 
\end{enumerate}
\end{definition}

Set the following hypotheses:
\begin{hypotheses}\label{H: gen fine G-graph}
Let $(G,\mc{P})$ be a relatively hyperbolic pair and suppose that $G$ acts on a connected hyperbolic graph $\Gamma$ so that $\Gamma$ is generalized fine with respect to the action of $(G,\mc{P})$.

For $P\in\mc{P}$ and $g\in G$, let $\Gamma_{P^g}$ be the sub-graph stabilized by $P^g$.
Let $H\le G$ and let $\Gamma_H$ be a quasi-convex cocompact core for $H$ in $\Gamma$. If $\Gamma$ is not fine (only generalized fine), we make the following additional assumption: 
\begin{equation}\label{item:assump} 
\mbox{If $\Gamma_H\cap \Gamma_{P^g} \ne\emptyset$ then $|H\cap P^g| =\infty$.} \tag{$\dagger$}
\end{equation}
\end{hypotheses}

Here is a rough outline of the proof of Theorem~\ref{Thm: rqc criterion}: we prove that the action of $H$ on $\Gamma_H$ implies $H$ is hyperbolic relative to a finite collection of vertex stabilizers $\mc{D}$. Then $H$ admits a geometrically finite convergence group action on the Bowditch boundary of $(H,\mc{D})$. 
We then show that the inclusion $\Gamma_H\to \Gamma$ induces an equivariant inclusion on Bowditch boundaries whose image is the limit set of $H$ so that the induced action of $H$ on its limit set in the Bowditch boundary of $(G,\mc{P})$ is a geometrically finite convergence group action. 
We now recall Yaman's dynamical characterization \cite{Yaman} of relative hyperbolicity:
\begin{definition}[{As stated in \cite[Definition 3.1 (RH-1)]{Hruska2010}}]\label{D: dynamic rel hyp}
Suppose $(G,\mc{P})$ has a geometrically finite convergence group action on a compact, metrizable space $M$, Then $(G,\mc{P})$ is a {relatively hyperbolic pair}.
\end{definition}
Yaman also proves (see, for example, \cite[Theorem 5.2]{Hruska2010}) that the space $M$ from Definition~\ref{D: dynamic rel hyp} is equivariantly homeomorphic to the Bowditch boundary of the pair $(G,\mc{P})$.

For details about geometrically finite actions, see \cite[Section 3.1]{Hruska2010}. If $H\le G$, recall that the \textbf{limit set of $H$} in $M$, denoted $\Lambda H$, is the smallest closed $H$--invariant subset of $M$.  
We use Definition~\ref{D: dynamic rel hyp} in conjunction with the following definition for relative quasi-convexity:
\begin{definition}[{\cite[Definition 6.2 (QC-1)]{Hruska2010}}]\label{QC1}
Let $(G,\mc{P})$ be a relatively hyperbolic group that acts on a compact metrizable space as a geometrically finite convergence group. A subgroup $H\le G$ is \textbf{relatively quasi-convex} if the induced convergence action of $H$ on the limit set $\Lambda H\subseteq M$ is geometrically finite. 
\end{definition}

\begin{proposition}
Assume Hypotheses~\ref{H: gen fine G-graph}.
Let $\mc{D}$ be a (finite) collection of $H$--conjugacy representatives of vertex stabilizers for the action of $H$ on $\Gamma_H$.
Then $(H,\mc{D})$ is a relatively hyperbolic pair. 
\end{proposition}

\begin{proof}
If $\Gamma$ is fine then $\Gamma_H$ is fine, so by Definition~\ref{D:  fine relhyp}, $(H,\mc{D})$ is a relatively hyperbolic pair. 

Otherwise, let $\Gamma_0$ be the stable part of $\Gamma$. 
Since $\Gamma_H\cap \Gamma_{P^g}\ne \emptyset$ implies $H\cap P_g$ is infinite by \eqref{item:assump}, the edges of $\Gamma_H$ with finite stabilizer in $H$ are precisely those that have finite stabilizer in $G$. It is now straightforward to verifry that $\Gamma_H$ is generalized fine with respect to the $H$--action. 
By Proposition~\ref{P: gen fine rel hyp}, $(H,\mc{D})$ is a relatively hyperbolic pair. 
\end{proof}

For clarity and completeness, we repeat Bowditch's construction of the Bowditch boundary from a fine hyperbolic graph: 
\begin{definition}[{\cite[Section 9]{BowditchRH}}]\label{D: fine boundary}
Let $(G,\mc{P})$ be a relatively hyperbolic pair and suppose that $\Gamma$ is a graph that is generalized fine with respect to the action of $(G,\mc{P})$. 
Let $\Gamma'$ be a complete electrification of $\Gamma$ with respect to the $\Gamma_{P^g}$ and for all $P\in\mc{P}$ and $g\in G$, let $v_{P^g}$ be the vertex of $\Gamma'$ stabilized by $P^g$.
Let $\triangle \Gamma' = \partial\Gamma' \sqcup V(\Gamma')$ endowed with the following topology:
If $A$ is any finite subset of the vertices of $\Gamma'$ and $a\in \triangle \Gamma'$, define $N(a,A)$ to be the set of $b\in \triangle \Gamma'$ so that every geodesic from $a$ to $b$ avoids $A\,\setminus\, a$. A subset $U\subseteq \triangle \Gamma'$ is open if for every $a\in U$, there exists a finite set of vertices $A\subseteq V(\Gamma')$ so that $N(a,A)\subseteq U$.

Define $\Pi_{\Gamma'} = \{v_{P^g}:\,P\in\mc{P},\,g\in G\}$, the \textbf{peripheral points} of the Bowditch boundary. 

Let $\partial_B\Gamma' = \partial \Gamma'\cup \Pi_{\Gamma'}$ where $\partial \Gamma'$ is the visual boundary of the hyperbolic graph $\Gamma'$. We  refer to the points of $\partial\Gamma'$ as the \textbf{conical limit points} of the Bowditch boundary. 
The topology on $\partial_B\Gamma'$ is the subspace topology induced by the topology on $\triangle \Gamma'$.   
\end{definition}

\begin{remark}
In this section, we explicitly use the notation $\partial_{B}\Gamma'$ to denote the construction of the Bowditch boundary of $(G,\mc{P})$ from the graph $\Gamma'$. Outside of Section~\ref{S: rqc technical}, we use the notation $\partial_{\mc{P}}G$ to refer to the Bowditch boundary of $G$ with respect to $\mc{P}$. 
\end{remark}

 Assuming Hypotheses~\ref{H: gen fine G-graph}, let $\Gamma_H'$ be the image of $\Gamma_H$ in $\Gamma'$. 
Following Definition~\ref{D: fine boundary}, we can define $\partial_B \Gamma_H'$ and $\partial_B \Gamma_H'$ embeds in $\partial_B \Gamma'$.
\begin{proposition}\label{P: subbdd closed}
Assuming Hypotheses~\ref{H: gen fine G-graph}, $\partial_B \Gamma_H'$ is closed in $\partial_B\Gamma'$. 
\end{proposition}

We prove Proposition~\ref{P: subbdd closed} by showing that the complement of $\partial_B \Gamma_H'$ is open in $\partial_B \Gamma'$. 
Specifically, if $y\in \partial_B\Gamma'\,\setminus\, \partial_B\Gamma_H'$ we find an open neighborhood of $y$ that does not contain any points of $\partial_B\Gamma_H'$. 
For the topology introduced in Definition~\ref{D: fine boundary}, it suffices to prove that there exists a finite set of vertices that any geodesic from $y$ to a point in $\partial_B\Gamma_H'$ must pass through.

\begin{proof}
Since $\Gamma_H$ is convex in $\Gamma$, $\Gamma'_H$ is $s$--quasi-convex in $\Gamma'$ for some $s\ge 0$. 
Set $\delta>1$ so that $\Gamma_H'$ has $\delta$--thin triangles.

Let $y\in \partial_B \Gamma'\,\setminus\,\partial_B \Gamma_H'$.
To prove that $\partial_B\Gamma' \,\setminus\, \partial_B\Gamma_H'$ is open, we find a finite $A_y\subseteq V_\Gamma$ so that if $h\in \partial_B \Gamma_H'$, any geodesic from $y$ to $h$ passes through $A_y$. If so, then $N(y,A_y)\subseteq \partial_B\Gamma'\,\setminus\, \partial_B \Gamma_H'$. We now fix a geodesic $\gamma$ from $y$ to some $h\in \partial_B\Gamma_H'$ and split the proof into two cases depending on whether $y$ is a conical limit point or a peripheral point.

\begin{figure}
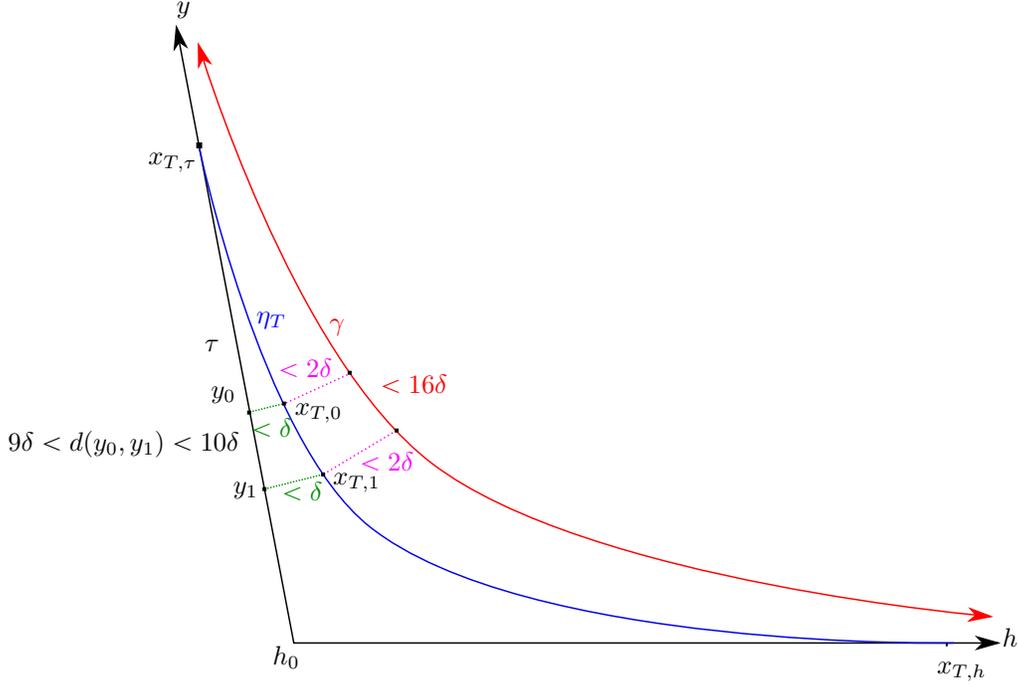
\caption{The situation in Proposition~\ref{P: subbdd closed} in Case 1 where $h\notin \Pi_{\Gamma'}$.}
\end{figure}

\textbf{Case 1: $y\in \partial\Gamma'$.} 
 Fix a base point $h_0\in \Gamma_H'$. 
Let $\tau$ be a geodesic ray from $h_0$ to $y$ and choose a vertex $y_0\in \tau$ so that $d(y_0,\Gamma_H')>2s+100\delta$.
Choose a second vertex $y_1\in V(\Gamma')$ on $\tau$ so that $9\delta<d(y_0,y_1) <10\delta$ and $y_0$ lies between $y_0$ and $h_0$.

Let 
\[A_y = \{v\in V(\Gamma'):\, v \text{ lies on an arc from $y_0$ to $y_1$ of length at most $22\delta$}\}\]
By \cite[Proposition 2.1 (F2)]{BowditchRH} and the fineness of $\Gamma'$, there are only finitely many arcs from $y_0$ to $y_1$ of length at most $22\delta$, so $A_y$ is finite.
We show that $\gamma$ intersects $A_y$.

We claim that $d(y_0,\gamma),d(y_1,\gamma)<3\delta$. 
First suppose $h\in \partial \Gamma'$, so we can parameterize $\gamma:(-\infty,\infty)\to \Gamma'$ where $\lim_{t\to \infty}\gamma(t) = y$ and $\lim_{t\to -\infty} \gamma(t) = h$. 
There exists $M$ so that for $t>0$ with $t$ sufficiently large, we have both $d(\gamma(t),\tau)<M$ and $d(\gamma(-t),\Gamma_H') <M$. 
Then there are geodesics $\eta_t$ whose endpoints are $x_{t,\tau}\in \tau$ and $x_{t,h}\in\Gamma_h'$ with $d(x_\tau,\gamma)<M$ and $d(x_h,\gamma)<M$. 
We claim that there are $x_{t,0},x_{t,1}\in \eta_t$ so that $d(y_0,x_{t,0})<\delta$ and $d(y_1,x_{t,1})<\delta$.
Indeed, when $t$ is large, there is a geodesic triangle with vertices $h_0,x_{t,\tau},x_{t,h}$ so that one side is $\eta_t$, one side lies in $\tau$ and the other side lies in $\mc{N}_s(\Gamma_H')$. 
Since triangles are $\delta$--thin and $d(y_0,\Gamma'_H),d(y_1,\Gamma'_H)>2s+90\delta$, $d(y_0,\eta_t),d(y_1,\eta_t)<\delta$.

Again, for large enough $t>0$, each of the quantities
\[
d(x_{t,\tau},x_{t,0}), \quad d(x_{t,\tau},x_{t,1}), \quad d(x_{t,h},x_{t,0}), \quad d(x_{t,h},x_{t,1})
\]
can be made arbitrarily large.  In particular, they can all be made to exceed the constant $M+2\delta$. 
A standard hyperbolic geometry argument using a $2\delta$-slim quadrilateral then implies that $d(x_{t,0},\gamma),d(x_{t,1},\gamma)<2\delta$. 
Then for appropriately large $T>0$, we have points $z_{T,0},z_{T,1}\in\gamma$ so that $d(z_{T,0},x_{T,0}),d(z_{T,1},x_{T,1})<2\delta$ which implies that $d(y_0,z_{T,0})<3\delta$ and $d(y_1,z_{T,1})<3\delta$. 

In the case that $h\in\Pi_{\Gamma'},$ a similar argument where $x_{t,h}$ is replaced by $h$ proves the claim that $d(y_0,\gamma),d(y_1,\gamma)<2\delta$ and that we can choose $z_{T,0},z_{T,i}$ in $\gamma$ so that $d(y_i,z_{T,i})<2\delta$ for $i=0,1$.

Since $9\delta<d(y_0,y_1)<10\delta$, then $3\delta<d(z_{T,1},z_{T,2})<16\delta$ by the triangle inequality.
Construct an arc $\sigma$ consisting of:
\begin{itemize}
\item a path of length at most $3\delta$ from $y_0$ to $z_{T,0}$,
\item a sub-path of $\gamma$ from $z_{T,0}$ to $z_{T,1}$ whose length is more than $1<3\delta$ but has length at most $16\delta$,
\item a path of length at most $3\delta$ from $y_1$ to $z_{T,1}$.
\end{itemize} 
The sub-path of $\gamma$ is long enough that it contains a vertex of $\Gamma'$.
Therefore, $\sigma$ is a path between $y_0$ and $y_1$ with length at most $22\delta$ and intersects $\gamma$ in a vertex.

\textbf{Case 2: $y\in \Pi_{\Gamma'}$.} 
Fix a base point $h_0\in\Gamma_H'$ and let $\tau$ be a geodesic from $h_0$ to $y$. 
Let $y_0 = y$. Let $y_1\in V(\Gamma')$ be a vertex on $\tau$ such that $\delta< d(y_0,y_1)< 2\delta$ or if no such vertex exists, set $y_1=h_0$. 
Let 
\[A_y = \{v\in V(\Gamma'):\,\text{$v$ lies on a path of length at most $10\delta$ from $y_0$ to $y_1$ }\}.\]
As in the proof of Case 1, $A_y$ is finite by \cite[Proposition 2.1 (F2)]{BowditchRH}.
Let $h\in \partial_B \Gamma_H'$ and suppose $\gamma$ is a geodesic from $y$ to $h$. 
Let $z\in \gamma$ be a vertex of $\Gamma_H'$ so that $\delta<d(z,y)<2\delta$, or if no such vertex exists, then $h\in\Pi_{\Gamma'}$ and we can set $z=h$ with $0<d(z,y)<\delta$. 

Consider a geodesic triangle with vertices $z,y,h_0$ and let $\eta$ be the side joining $z$ to $h_0$. 
Since $d(y,z)<2\delta$ there must exist some point $x\in \eta$ and $y_2\in\tau$ with $d(x,y_2)<\delta$ and $d(x,z)<2\delta$ by $\delta$--thinness of geodesic triangles. 
Then $d(y,y_2)<5\delta$ and $\tau$ is geodesic, so $d(y_2,y_1)<5\delta$. 
Therefore, there is a path from $y_0$ to $y_1$ of length at most $10\delta$ that intersects $\gamma$ in the vertex $z$, and so $\gamma\cap A_y\ne\emptyset$. 

We have showed that for all $y\in \partial_B\Gamma'\,\setminus\, \partial_B\Gamma_H'$, there exists a finite collection of vertices $A_y$ so that $N(y,A_y)\subseteq \partial_B\Gamma'\,\setminus\, \partial_B\Gamma'_H$. Thus $\partial_B\Gamma'\,\setminus\, \partial_B\Gamma'_H$ is open in $\partial_B\Gamma'$. 
\end{proof}


\begin{proposition}
The limit set of $H$, $\Lambda H$, in $\partial_B\Gamma'$ is $\partial_B\Gamma_H'$. 
\end{proposition}

\begin{proof}
Immediately, $\partial_B \Gamma'_H\subseteq \Lambda H$. 
By Proposition~\ref{P: subbdd closed}, $\partial_B\Gamma'_H$ is closed and $H$--invariant, so $\Lambda H\subseteq \partial_B \Gamma'_H$ because the limit set of $H$ is the smallest closed $H$--invariant subset of $\partial M$. 
\end{proof}

\qccriterion*

\begin{proof}
The action of $H$ on the fine hyperbolic sub-graph $\Gamma_H'$ constructed by electrifying with respect to the $\Gamma_{P^g}$ shows that $H$ is hyperbolic relative to the infinite vertex stabilizers. By \cite[Theorem 1.1]{Hruska2010}, the induced convergence group action of $H$ on $\partial_B \Gamma'_H$ is geometrically finite. 
Since $\partial_B\Gamma_H' = \Lambda H$, the subgroup $H$ has a geometrically finite convergence group action on $\Lambda H$ and hence $H$ satisfies Definition~\ref{QC1} for relative quasi-convexity. 
\end{proof}

We can also rephrase Theorem~\ref{Thm: rqc criterion} as a criterion for relative quasi-convexity in fine hyperbolic graphs. This special case recovers the following result of Martinez-Pedroza and Wise:

\finecriterion*

\begin{proof}[Proof Sketch]
One direction follows immediately from Theorem~\ref{Thm: rqc criterion}.
If $H$ is relatively quasi-convex, the join of $\Lambda H$ in $\Gamma$ provides the quasi-convex core for $H$, see \cite[end of Section 5]{BowditchRH}.
\end{proof}

\begin{remark}
For a relatively hyperbolic pair $(G,\mc{P})$, we use the notation $\partial_{\mc{P}}G$ to denote the Bowditch Boundary of $G$ with respect to $\mc{P}$. 
When $\Gamma'$ is a fine hyperbolic graph that witnesses the relative hyperbolicity of $(G,\mc{P})$, we henceforth conflate $\partial_{\mc{P}}G$ with $\partial_B \Gamma'$. 
\end{remark}

We can also prove that hyperplane stabilizers are relatively quasi-convex:
\begin{cor}\label{C: hypstab rel qc}
Let $(G,\mc{P})$ act relatively geometrically on a CAT(0) cube complex $\tilde{X}$. Let $H$ be the stabilizer of a hyperplane $W$ of $\tilde{X}$. 
Then $H$ is relatively quasi-convex in $(G,\mc{P})$. 
\end{cor}

\begin{proof}
Subdivide $\tilde{X}$ cubically once to a complex $\tilde{X}_W$ so that $W$ is a sub-complex. The action of $(G,\mc{P})$ on $\tilde{X}_W$ is still relatively geometric.  
Recall from Example~\ref{E: relgeom is genfine} that $\tilde{X}_W\oskel$ is  generalized fine with respect to the action of $(G,\mc{P})$. 
If $\Gamma_{P^g}$ intersects $W\cap \tilde{X}_W\oskel$, then a finite index subgroup of $P^g$ stabilizes an edge of $\tilde{X}$ dual to $W$. 
Hence $P^g\le \Stab_G (W)$.  
Since $W$ is convex and cocompact in $\tilde{X}_W$, $W \cap \tilde{X}_W\oskel$ is an $H$--invariant $H$--cocompact connected convex sub-graph of $\tilde{X}_W\oskel$. The relative quasi-convexity of $H$ now follows from Theorem~\ref{Thm: rqc criterion}.
\end{proof}

	\section{A Separation Criterion for the Bowditch Boundary}
	\label{S: separation criterion}
	
	A construction of Sageev \cite{Sageev95} shows that group actions on CAT(0) cube complexes arise naturally from groups with collections of `codimension-1' subgroups. Building on work of Bergeron and Wise \cite{BergeronWise} for hyperbolic cubulations, the first and second author gave a boundary criterion \cite[Theorem 2.6]{RelCannon} for relatively geometric actions that guarantees the existence of a relatively geometric action of a relatively hyperbolic pair $(G,\mc{P})$ on a CAT(0) cube complex whenever $G$ contains a sufficient collection of full relatively quasi-convex subgroups that separate points in the Bowditch boundary. 
	The main theorem of this section, Theorem~\ref{T: separation criterion} helps to show that stabilizers of quasi-convex cores in fine hyperbolic graphs that exhibit `hyperplane like' behavior and provide a source of codimension-1 subgroups that may be used with \cite[Theorem 2.6]{RelCannon}.

	\subsection{Hypersets and Hypercarriers}
	
	Let $\Gamma$ be a graph.
	A \textbf{hyperset} $L$ in $\Gamma$ is a collection of edge midpoints and vertices of $\Gamma$ so that $\Gamma\,\setminus\, L$ has two components. A \textbf{(hyperset) carrier} $J$ is the minimal sub-graph of $\Gamma$ containing $L$ that has the following property: if $v_1,v_2\in J$ and $v_1,v_2$ are joined by an edge in $\Gamma$, then $J$ contains the edge between $v_1,v_2$.
	
	Hypersets and carriers arise naturally in the one-skeleton of a CAT(0) cube complex.  We will see that they are particularly helpful in the setting of relatively geometric actions:
	\begin{example}\label{hypersets in CCC}
	Let $(G,\mc{P})$ act relatively geometrically on a CAT(0) cube complex $\tilde{C}$, and let $W$ be a hyperplane. 
Then $L = W\cap \tilde{C}^{(1)}$ is a hyperset. 
A hyperset carrier $J$ for $L$ is the intersection of the hyperplane carrier of $W$ with $\tilde{C}^{(1)}$.  	
	\end{example}
	In this situation, we refer to $L$ as the \textbf{hyperset associated to $W$} and $J$ as \textbf{the (hyper)carrier (of the hyperset associated to $W$)}.
	
	
	
	We observe the following useful fact in the setting of generalized fine hyperbolic graphs:
	\begin{obs}
	\label{P: collapsing hypersets}
	Let $\Sigma$ be generalized fine with respect to the action of $(K,\mc{D})$ and let $\Sigma_{D^k}$ be the sub-graph of cells whose stabilizer is commensurable to $D^k$ for $D\in\mc{D}$, $k\in K$. 
	Let $\sigma:\Sigma\to \Gamma$ be the electrification map that collapses the $\Sigma_{D^k}$. 
	Let $S = \cup \{\Sigma_{D^k}:\, \Sigma_{D^k}\cap L \ne \emptyset\}$. 
	If $L$ is a hyperset and $J$ is a quasi-convex hyperset carrier, then $\sigma(L)$ is a hyperset in $\Gamma$, $\sigma(J)$ is a hyperset carrier and the components of $\Gamma\,\setminus\, \sigma(L)$ are images of the components of $\Gamma\,\setminus\, (S\cup L)$.   
	\end{obs}

\subsection{The separation criterion: }

We set the following assumptions for the remainder of this subsection:
\begin{hypotheses}\label{separation setup}
Let $(K,\mc{D})$ be a relatively hyperbolic pair and let $K$ act on a fine $\delta$--hyperbolic graph $\Gamma$ with the following properties:
\begin{itemize}
\item The action of $K$ is cocompact,
\item edge stabilizers are finite, and
\item each $D\in\mc{D}$ stabilizes a single vertex. 
\end{itemize}
Let $L$ be a hyperset with connected quasi-convex carrier $J$. 
\end{hypotheses}

Our goal is to decide whether two points in the Bowditch boundary $\partial_{\mc{D}}K$ lie in complementary components of the limit set of $\Stab_G(L)$. 
The hyperset $L$ separates $\Gamma$ into two complementary components, but it is not immediately apparent that the limit set $\Lambda \Stab_G(L)$ partitions the Bowditch boundary into multiple components with respect to the topology described in Definition~\ref{D: fine boundary}.

\begin{definition}
	With the setup in Hypotheses~\ref{separation setup}, $J$ has the \textbf{two-sided carrier property} if there exist connected quasi-convex subsets $J^+$ and $J^-$ so that 
	\[
		J^+\cap J^- \, \subseteq \,  L \,  \subseteq \,  J^+\cup J^-,
	\]
	where $J^+\cap J^- =: J$ is the hypercarrier of $L$.
	every path in $\Gamma$ between vertices in the two distinct components of $\Gamma\,\setminus\, L$ must intersect both $J^+$ and $J^-$ and if $v$ is a vertex of $J^-\cap J^+$ with infinite stabilizer, then $v\in \Lambda\Stab_K(L)$.  
	\end{definition}
	
	The two-sided carrier property arises naturally in our intended application to relatively geometric actions. 
	
	\begin{example}
	When $(K,\mc{D})$ acts relatively geometrically on a CAT(0) cube complex $\tilde{X}$, and $W$ is a hyperplane with associated hyperset $L$ and carrier $J$, there are two natural sides $J^+$ and $J^-$ of the carrier $J$. 
	Note that the two sides are slightly larger than the combinatorial hyperplanes on either side of $W$ because $L$ needs to be contained in their union.
	Recall that $\tilde{X}\oskel$ usually fails to be fine, but Proposition~\ref{P: electric is fine} implies that collapsing compact sub-graphs of $\tilde{X}\oskel$ yields a fine hyperbolic graph.  
	The images of $L$ and $J$ remain a hyperset and hypercarrier respectively, but $J^+$ and $J^-$ may have overlapping images. 
	As we will see, this may only happen at vertices that are already parabolic points in the limit set of the stabilizer of $W$. 
	\end{example}
	
	Recall that under Hypotheses~\ref{separation setup}, the points of the Bowditch boundary are either \textbf{conical limit points} that lie in $\partial \Gamma$, the visual boundary of $\Gamma$ or are \textbf{parabolic vertices} (also called \textbf{peripheral vertices}) of $\Gamma$, which are stabilized by maximal parabolics. 
	
\begin{definition}\label{D: separation}
Assuming Hypotheses~\ref{separation setup}, we say that $L$ \textbf{separates $x,y\in\partial_{\mc{D}}K$} if $x,y\notin \Lambda \Stab_K(L)$ and one of the following holds:
	\begin{itemize}
	\item $x,y$ are both conical limit points, and there exists some geodesic $\gamma:(-\infty,\infty)\to \Gamma$ with $\lim_{t\to \infty}\gamma(t) = x$ and $\lim_{t\to-\infty}\gamma(t) = y$ so that there exists $T>0$ so that for all $t_- < -T <0 <T<t_+$, $\gamma(t_+)$ and $\gamma(t_-)$ are in distinct components of $\Gamma\,\setminus\, L$, 
	\item $x$ is a parabolic vertex in $\Gamma$, $y$ is a conical limit point, and there exists some geodesic $\gamma:[0,\infty)\to \Gamma$ from $x=\gamma(0)$ to $y$ so that for $t$ sufficiently large, $\gamma(t)$ and $x$ are in distinct components of $\Gamma\,\setminus\, L$,
	\item $x,y$ are both parabolic vertices in $\Gamma$, and $x,y$ lie in distinct components of $\Gamma\,\setminus\, L$. 
	\end{itemize}
	In the first two cases, we say that $\gamma$ \textbf{witnesses that $L$ separates $x$ and $y$}. 
	\end{definition}


The remainder of this section is devoted to proving Theorem~\ref{T: separation criterion}:
\begin{theorem}\label{T: separation criterion}
Assume Hypotheses~\ref{separation setup} and assume the setup satisfies the two-sided carrier property. If $L$ separates $x,y\in \partial_{\mc{D}}K\,\setminus\, \Lambda\Stab_K(L)$, then there exists a subgroup ${K_L} \le \Stab_K(L)$ of index at most 2 so that $x,y$ are in ${K_L}$--distinct components of $\partial_{\mc{D}}K\,\setminus\, \Lambda {K_L}$. 
\end{theorem}

Recall from Definition~\ref{D: fine boundary} that if $x\in\partial_{\mc{D}}K$ and $A$ is a set of vertices in $\Gamma$, the set $N(x,A)$ consisting of $y\in \partial_{\mc{D}}K$ so that some geodesic from $x$ to $y$ avoids $A\,\setminus\, \{x\}$ is an open neighborhood of $x$ in $\partial_{\mc{D}}K$. The following lemma helps us control certain open neighborhoods of points in $\partial_{\mc{D}} K\,\setminus\, \Lambda \Stab_K(L)$. 
\begin{lemma}\label{L: coarse projection rel Bowditch}
Asssume Hypotheses~\ref{separation setup}. Let $J_0$ be a quasi-convex subset of $J$, and  
let $x\in \partial_\mc{D} K \,\setminus\, \Lambda \Stab_K(J)$ so that $x\notin J_0$. There exists a finite collection of vertices $V_x$ so that for any vertex $j\in J_0$, and any geodesic $\gamma_{x,j}$ between $x$ and $j$, the intersection $\gamma_{x,j} \cap (V_x\,\setminus\, x)$ is not empty.    
\end{lemma}

\begin{proof}
Fix some $j_0\in J_0$, and fix $R$ so $J_0$ is $R$--quasi-convex.
Consider a (possibly ideal) geodesic triangle with vertices $x,j_0,j$ and sides $\gamma_{x,j},\gamma_{x,j_0},\gamma_{j,j_0}$. 

\textbf{Case 1: $x$ is a conical limit point. }

Let $x_0$ be a vertex on $\gamma_{x,j_0}$ so that $d(x_0,J_0)> R+\delta+1$. Such a vertex exists because $x\notin\Lambda \Stab_K(J)$.  
Since $\gamma_{j,j_0}\subseteq \mc{N}_R(J_0)$ by quasi-convexity, $d(x_0,\gamma_{j,j_0})>\delta$. 
Therefore, by hyperbolicity, there exists a vertex $y_0\in \gamma_{x,j}$ so that $d(x_0,y_0)<\delta$. 
Then there exists vertices $y_1\in \gamma_{x,j}$ and $x_1\in \gamma_{x,j_0}$ so that $y_1$ lies between $y_0$ and $j$ on $\gamma_{x,j}$, $x_1$ lies between $x_0$ and $j_0$ on $\gamma_{x,j_0}$, and $d(y_1,x_1)\le \delta$.
Hence there exists a $(1,4\delta)$--quasi-geodesic arc from $j_0$ to $x_0$ following:
\begin{itemize}
\item $j_0$ to $x_1$ via $\gamma_{x,j_0}$,
\item $x_1$ to $y_1$ via a geodesic of length at most $\delta$ 
\item $y_1$ to $y_0$ via $\gamma_{x,j}$
\item $y_0$ to $x$ via a geodesic of length at most $\delta$
\end{itemize}  
Then $y_1$ lies on a $(1,4\delta)$--quasi-geodesic arc between $x_0$ and $j_0$. 
Let $V_{x_0,j_0}$ be the sub-graph of $(1,4\delta)$--quasi-geodesic arcs between $x_0$ and $j_0$. By \cite[Lemma 8.2]{BowditchRH}, $V_{x_0,j_0}$ is locally finite. The length of any such arc is uniformly bounded above by $d(x_0,j_0)+4\delta$, so this sub-graph has finite diameter and is therefore finite. 
Hence $y_1$ is one of finitely many vertices in $\Gamma$. 

\textbf{Case 2: $x$ is a parabolic point. }

If $d(x,J_0)> R+\delta+1$, carry out the same proof as in the preceding case (the choice $x_0 = x$ suffices).

Hence assume $d(x,J_0)\le R +\delta+1$.   
The geodesic triangle with vertices $x,j,j_0$ is $\delta$--thin. Therefore, there exists a vertex $j_1\in J_0$ and a sub-path of $\gamma_{x,j}$ of length at most $R+\delta+2$ between $x$ and a vertex $y_1\in \gamma_{x,j}$ so that $y_1\ne x$ and $y_1,j_1$ have the following properties:
\begin{itemize}
\item either $y_1$ is the vertex on $\gamma_{x,j}$ in $J_0$ that is closest to $x$ in which case we set $j_1=y_1$, or
\item $d(y_1,j_1) < R + \delta+1$, a shortest path from $y_1$ to $j_1$ does not backtrack along $\gamma_{x,j}$ and there is an arc from $x$ to $j_1$ passing through $y_1$ that does not contain any vertices of $J_0$ other than $j_1$.
\end{itemize}  
In the first case, $y_1\ne x$ because $x\notin J_0$. 
In the second case, one might worry that eliminating backtracking could force us to choose $y_1=x$. 
If $d(x,j)\ge R+\delta+1$, this is not a problem. If $d(x,j)< R+\delta+1$, we can ensure we are in the first case by letting $y_1=j_1$ be the vertex on $\gamma_{x,j}$ in $J_0$ that is closest to $x$. 
In both cases, we obtain an arc $\sigma$ from $x$ to $j_1$ that contains $y_1$ and no vertices of $J_0$ other than $j_1$. 

In all of these above cases: 
\[d(j_1,j_0)\le d(j_0,x)+d(j_1,y_1)+d(y_1,x)\le 3R+3\delta+4.\]

Since $J_0$ is connected and quasi-convex, and $\Gamma$ is hyperbolic there exist $\lambda\ge 1$ and $\epsilon \ge 0$ so that some $(\lambda,\epsilon)$--quasi-geodesic arc $\rho$ in $J_0$ connects $j_0$ and $j_1$.
Thus $\rho$ has length at most $\lambda(3R+3\delta+4) +\epsilon$. 
Note that $\sigma$ cannot backtrack along $\rho$ at $j_1$ because every edge of $\rho$ has both endpoints in $J_0$ while $j_1$ is the only vertex on $\sigma$ that lies in $J_0$. 
Hence $y_1$ is on an arc from $x_0$ to $j_0$ of length at most 
\[\lambda(3R+3\delta+4) +\epsilon +d(j_1,y_1)+d(y_1,x)\le \lambda(3R+3\delta+4)+\epsilon +3R+3\delta+4\]
 and there are finitely many such arcs by \cite[Proposition 2.1 (F2)]{BowditchRH}, since $\Gamma$ is fine. 
Then there are finitely many possibilities for $y_1$. 
\end{proof}

	\begin{proposition}\label{separation divider}
	Suppose $L$ separates $x,y\in \partial_{\mc{D}}K\,\setminus\, \Lambda \Stab_K(L)$. 
	If $z\in\partial_{\mc{D}}K\,\setminus\, \Lambda \Stab_K(L)$ and $L$ does not separate $x,z$, then $L$ separates $y,z$. 
	\end{proposition}
	
	\begin{proof}
		Consider a geodesic triangle with vertices $x,y,z$ and sides $\gamma_{xy},\gamma_{yz},\gamma_{xz}$ where the ordered subscripts indicate the endpoints and orientation. Further, assume $\gamma_{xy}$ witnesses that $L$ separates $x,y$. 
		
		If $z$ is a conical limit point then all but a finite length of $\gamma_{xz}$ lies in a single component $C$ of $\Gamma\,\setminus\, L$ because $L$ does not separate $x$ from $z$.  By hypothesis $z \notin \Lambda \Stab_K(L)$, so, as $t \to \infty$, the quasi-convexity of $L$ ensures that $d(\gamma_{xz}(t),L)\to \infty$.  Moreover, for all $t > 0$ large enough $\gamma_{yz}(t)$ lies in $C$ by hyperbolicity.  
		If $y$ is a parabolic point then $y$ lies in the other component $C'\ne C$ of $\Gamma\,\setminus\, L$, otherwise $y$ is a conical limit point and, for all $t > 0$ sufficiently large, $\gamma_{zy}(t) = \gamma_{yz}(-t)$ lies in $C'$.  In either case, $L$ separates $y$ from $z$.  
		
		If $z$ is a parabolic vertex in $\gamma$, then $z \in C$ by hypothesis.  As above, whether $y$ is a parabolic or conical limit point, $L$ separates $y$ from $z$.
	\end{proof}
	
	By Proposition~\ref{separation divider}, there is an equivalence relation $\sim$ on $\partial_\mc{D}K \,\setminus\, \Lambda \Stab_K(L)$ defined by $x\sim y$ if and only if $x=y$ or $L$ does not separate $x,y$. There are two equivalence classes. 
	
	\begin{proposition}\label{P: separation criterion}
	If $x,y\in\partial_{\mc{D}}(K)\,\setminus\, \Lambda\Stab_K(L)$ and $x\not\sim y$, then $x,y$ lie in distinct components of $\partial_{\mc{D}}K\,\setminus\, \Lambda\Stab_K(L)$. By passing to an index at most $2$ subgroup ${K_L}$ of $\Stab_K(L)$, these components are ${K_L}$--distinct. 
	\end{proposition}
	
	\begin{proof}
	We claim that if $z\in \partial_{\mc{D}}K \,\setminus\, \Lambda \Stab_K (L)$, then there exists an open neighborhood $U$ of $z\in \partial_{\mc{D}}K \,\setminus\, \Lambda \Stab_K (L)$ so that $u\sim z$ for all $u\in U$. 
	
To this end, we claim that there is a finite set of vertices $V_z$ so that any geodesic from $z$ that crosses $L$ must pass through $V_z$. 

If $z\notin J$, then Lemma~\ref{L: coarse projection rel Bowditch} with $J=J_0$ immediately provides $V_z$. 
Hence, we may assume $z$ is a parabolic vertex in $J$. 

The two-sided carrier property ensures that $J=J^+\cup J^-$ and $J^+\cap J^-\subseteq L$. Then $z\notin L$, so $z\in J^+\,\setminus\, J^-$ or $z\in J^-\,\setminus\, J^+$. Up to relabeling, we may assume $z\in J^+\,\setminus\, J^-$.  
 Since every path from $z$ to a vertex of the other component of $\Gamma\,\setminus\, L$ must pass through $J^-$, we can apply Lemma~\ref{L: coarse projection rel Bowditch} to show that there exists a finite set of vertices $V_z$ so that any geodesic $\gamma$ from $z$ to $J^-$ has $\gamma\cap V_z\ne \emptyset$.
 
Hence in all cases, if $L$ separates $z,w\in \partial_{\mc{D}}(K)\,\setminus\, \Lambda\Stab_K(L)$, any geodesic between $z$ and $w$ passes through $V_z$. 

Recall the set $N(z,V_z)$ from Definition~\ref{D: fine boundary} is an open neighborhood of $z$ in $\partial_{\mc{D}}(K)$. 
Therefore, if $U= N(z,V_z)$, then $z\sim u$ for all $u\in U$.

Thus we conclude that $[x]$ and $[y]$ are unions of components of $\partial_{\mc{D}} K\,\setminus\, \Stab_K(L)$ because they are open and partition $\partial_{\mc{D}} K\,\setminus\, \Stab_K(L)$.

Each $k\in \Stab_K(L)$ permutes the two components of $\Gamma\,\setminus\, L$, so $\Stab_K(L)$ acts on the equivalence classes of $\sim$. By passing to an index $2$ subgroup $K_L$ of $\Stab_K(L)$ if necessary, we can ensure that for any $k_L\in K_L$, $k_L\cdot[x]\ne [y]$, so $x$ and $y$ are in $K_L$--distinct components of $\partial_{\mc{D}} K\,\setminus\, \Stab_K(L)$. 
	\end{proof}

	Proposition~\ref{P: separation criterion} completes the proof of Theorem~\ref{T: separation criterion}.

\section{Separating points in the Bowditch boundary of a group acting relatively geometrically using hyperplane stabilizers}
\label{S: relgeom separation}
	
	For this section, let $(K,\mc{D})$ act relatively geometrically on a CAT(0) cube complex $\tilde{X}$. For $D\in\mc{D}$ and $k\in K$, let $\Sigma_{D^k}$ denote the complete sub-graph containing the vertices whose stabilizers are commensurable to $D^k$. Let $\Gamma$ be the complete electrification of $\tilde{X}$ with respect to 
	\[\{\Sigma_{D^k}:\,D\in\mc{D},\,k\in K\}.\] 
	Recall the electrification map $\beta: \tilde{X}^{(1)} \to \Gamma$ that collapses the $\Sigma_{D^k}$ to a single vertex is a continuous coarse inverse of the map in Proposition~\ref{P: electric qi}, and is $K$--equivariant. 
	
	Let $L$ be the hyperset associated to a hyperplane $W$ of $\tilde{X}$ as in Example~\ref{hypersets in CCC}. Let $J$ be the associated hyperset carrier. 
	Then $\beta(J)$ is a quasi-convex subset of $\Gamma$. Let $K_W$ be the stabilizer of $W$. 
	
	Since $J$ is associated to a hyperplane, $J\,\setminus\, L$ has two distinct components $J^+$ and $J^-$ that are connected and convex. Therefore $\beta(J^+)$ and $\beta(J^-)$ are both quasi-convex. 
	
	\begin{proposition}\label{P: hyperplane parabolic iff subgraph intersects}
	Let $W$ be a hyperplane of $\tilde{X}$, and let $D^k$ be a maximal peripheral subgroup. $W$ is dual to an edge of $\Sigma_{D^k}$ if and only if $D^k$ is commensurable to a subgroup of $\Stab_K (W)$. 
	\end{proposition}
	
	\begin{proof}
	If $W$ is dual to an edge $e$ of $\Sigma_{D^k}$, then $\Stab_K(e)$ is commensurable to $D^k$. Since $W$ is dual to $e$, $\Stab_K(e)\le \Stab_K(W)$. 
	
	Conversely suppose $D^k$ is commensurable to a subgroup $D_0\le \Stab_K(W)$ that stabilizes some vertex $x\in \Sigma_{D^k}$. 
	Let $\bar{W}$ be the carrier of $W$.

	Let $v$ be any vertex of $\tilde{X}\oskel$, let $s =\min\{d_{\tilde{X}\oskel}(z,v),\, z\in W\cap \tilde{X}\oskel\}$ and let
	\[M = \{z\in W\cap \tilde{X}\oskel:\, d_{\tilde{X}\oskel}(z,v) = s\}.\]  
	
	Assume for now that $M$ is finite.
	Let $w\in W\cap \tilde{X}\oskel$ be an edge midpoint that minimizes the combinatorial distance between $W\cap\tilde{X}\oskel$ and $x$, and let $e_W$ be the edge containing $w$.
Then $\Stab_K(x)\cdot w$ is finite because $M$ is finite. 
Hence $\Stab_K(w)$ and $\Stab_K(e_W)$ are commensurable to $\Stab_K(x)$. 
Then, $\Sigma_{\Stab_K(x)}$ intersects $W$. Therefore, it suffices to show that $M$ is finite. 
	
We prove that $M$ is finite by induction on $s-\frac12.$ Suppose toward a contradiction that $M$ is infinite with elements $v_0,v_1,v_2,v_3\ldots v_m,\ldots$. 
	If $s=\frac12$, then $W$ self-intersects or self-osculates, which is impossible in a CAT(0) cube complex. 
	Now assume $s>\frac12$.
	By passing to an infinite subset, we may assume no two $v_i$ lie in the same $\Sigma_{D^k}$ because each $\Sigma_{D^k}$ is finite. 
	Therefore, up to reindexing, we may assume $v_0$ and $v$ are not in the same $\Sigma_{D^k}$. Fix $\sigma_0$, a combinatorial geodesic from $v$ to $v_0$. 
	Let $\sigma_i$ be  combinatorial geodesics from $v_i$ to $v$. 
	Since the $\sigma_i$ are geodesic, we may choose $\sigma_i$ so that $\sigma_i\cap \sigma_0$ is connected. Note $|\sigma_i|=s$ for all $i$.
	It also follows that $\sigma_0$ contains an edge with finite stabilizer since $v_0,v$ are not in the same $\Sigma_{D^k}$. 
	
	If infinitely many $\sigma_i\cap\sigma_0$ contain an edge $f_0$ adjacent to $v$, set $v'\ne v$ to be the vertex of $f_0$. 
	Then there are infinitely many points in $W$ that realize the minimal distance, $s-1$, between $v'$ and $W$, so the inductive hypothesis gives a contradiction.
	
		Thus we may assume that for $i\ge 1$, $\sigma_i\cap \sigma_0 = \{v\}$ because $\sigma_i\cap \sigma_0$ is connected. 		
		The minimality of $s$ ensures that $\sigma_i$ does not intersect $W$ except at $v_i$. 
		Since $W$ is a hyperplane, there exists a combinatorially geodesic path $\mu_i$ from $v_i$ to $v_0$ in $\bar{W}$ that does not intersect $\sigma_i$ or $\sigma_0$ except at $v_i,v_0$ and this path has length at most $2s$ because $|\sigma_i| = s$. 
	Hence for each $v_i$, there exists a loop $\rho_i$ based at $v$ that follows $\sigma_i,\,\mu_i,$ and $\sigma_0$ of length at most $4s$. Since $\sigma_i$ and $\sigma_0$ are disjoint for $i\ge 1$, $\rho_i$ is a circuit.
Recall that $\sigma_0$ must contain an edge that has finite stabilizer, so the existence of infinitely many $\rho_i$ contradicts the generalized fineness of $\tilde{X}^{(1)}$. 
It follows that $M$ is finite. 
	\end{proof}
	
	\Cref{P: hyperplane parabolic iff subgraph intersects} is particularly relevant for studying hypersets in $\Gamma$ the electrified fine hyperbolic graph.  In particular, peripheral points that lie in hypersets coming from images of hyperplanes in $\tilde{X}$ are visible in the subgroup structure of the hyperplane stabilizer.
	
	\begin{proposition}\label{P: peripheral crossover}
 If $x\in \beta(J^+)\cap \beta(J^-)$ is a vertex, then $x$ is a peripheral point whose stabilizer is commensurable to a subgroup of $\Stab_K(W)$. 
\end{proposition}

\begin{proof}
Since $J^+\cap J^-\cap\tilde{X}^{(0)} = \emptyset$, there exist distinct vertices $y_+\in J^+\,\setminus\, L$ and $y_- \in  J^-\,\setminus\, L$ so that $\beta(y_+)=x = \beta(y_-)$. Since $\beta$ is the map that collapses the $\Sigma_{D^k}$, $y_+$ and $y_-$ both lie in some $\Sigma_{D_y^{k_y}}$ where $D_y\in\mc{D}$ and $k_y\in K$. Any combinatorial path between different sides of $W$ contains an edge dual $W$, so the connected sub-graph $\Sigma_{D_y^{k_y}}$ must contain an edge $e$ dual to $W$. 
The stabilizer of the edge $e$ stabilizes $W$ and must be commensurable to $D_{y}^{k_y}$ by \Cref{P: hyperplane parabolic iff subgraph intersects}. 
\end{proof}
	
	\begin{proposition}\label{P: hyperplane two-sided carrier}
	The hyperset carrier $\beta(J)$ has the two-sided carrier property.
	\end{proposition}
	
	\begin{proof}
	We immediately see that $\beta(J) = \beta(J^+)\cup \beta( J^-)$. 
	If $s\in \beta(J^+)\cap \beta(J^-)$, then $s$ has stabilizer commensurable to a subgroup of $K_W$, so $\beta(\Sigma_D^k)$ is a vertex in $\Lambda \Stab_K(\beta(L))$ the limit set of the hyperset stabilizer. 
	Let $\rho$ be a path between components of $\Gamma\,\setminus\,\beta(L)$. Let $\ell\in \beta(L)\cap \rho$. Either $\ell$ is the midpoint of an edge whose endpoints are in $\beta(J^+)$ and $\beta(J^-)$ or $\ell$ is a vertex formed by collapsing an edge dual to a hyperplane.  In both cases $\ell\in \beta(J^+)\cap \beta(J^-)$. 
	Hence every path between vertices of $\Gamma\,\setminus\,\beta(L)$ intersects both $\beta(J^+)$ and $\beta(J^-)$. 
	\end{proof}
	
	


	
We are now ready to prove that any two points $x,y\in\partial_{\mc{D}}K$ can be separated by a hyperset associated to a hyperplane

As we have seen, relatively geometric actions on CAT$(0)$ cube complexes let us take advantage of both the cubical geometry of $\tilde{X}$ as well as the fine graph structure of $\Gamma$.  The following \Cref{D: upstairs separation} lets us study separating hypersets using the separating properties of hyperplanes.  Note that it need not be the case that every hyperset in $\Gamma$ is the image of a hyperplane in $\tilde{X}$.

\begin{definition}\label{D: upstairs separation}
	Let $W$ be a hyperplane in $\tilde{X}$ with associated hyperset $L$ as in Example~\ref{hypersets in CCC}.  We say that $W$ separates $x,y\in \partial_{\mc{D}}K\,\setminus\, \Lambda \Stab_K(W)$ if $\beta(L)$ separates $x,y$ in the sense of \Cref{D: separation}. 
\end{definition}
	
We are now ready to prove that any two points $x,y\in\partial_{\mc{D}}K$ can be separated by a hyperset (associated to a hyperplane).
Our strategy is to show that if $x,y\in \partial_{\mc{D}}K$, then some hyperplane $W$ separates $x,y$ in the sense of Definition~\ref{D: upstairs separation}. 
Then we apply Theorem~\ref{T: separation criterion} to show that $x,y$ lie in distinct complementary components of $\Lambda\Stab_K (W)$.  
	
	\begin{lemma}\label{L: stabilizer classification}
	Let $x$ be a parabolic point in $\partial_{\mc{D}}K$ with stabilizer $K_x$ and let $\gamma$ be a combinatorial geodesic with end--vertices $v,w$ in $\tilde{X}$ so that $v\in \Sigma_{K_x}$, and $\gamma$ is of minimal length among all combinatorial geodesics between $w$ and $\Sigma_{K_x}$. Then every hyperplane dual to an edge of $\gamma$ does not intersect $\Sigma_{K_x}$. 
	\end{lemma}
	
	\begin{proof}
	Suppose $e$ is the edge of $\gamma$ with endpoint $v\in\Sigma_{K_x}$. By minimality, $e\not\subseteq \Sigma_{K_x}$. 
	Let $W$ be the hyperplane dual to $e$. 
	If $W$ intersects $\Sigma_{K_x}$, then $W$ is dual to an edge $f\in \Sigma_{K_x}$. Note that $\Stab_K(f)$ and $\Stab_K(v)$ are commensurable, so there exists a $k\in K_x$ so that $K$ fixes $f$, $v$ and not $e$. Since $k$ fixes the dual edge $f$, $k\cdot W = W$. Then $k\cdot e$ is adjacent to $v$ and is dual to $W$. Therefore the hyperplane $W$ self-osculates which is impossible in a CAT(0) cube complex (see for example \cite[Pages 20-21]{WiseBook}).
	
	The first paragraph shows that the first edge (counting from $v$) of any geodesic between $v$ and $w$ cannot be dual to a hyperplane that intersects $\Sigma_{K_x}$. 
	We now assume that the first $i$ edges of any minimal geodesic $\rho$ between $v$ and $w$ is dual to a hyperplane that does not intersect $\Sigma_{K_x}$ and prove that the $i+1$st edge of $\rho$ is dual to a hyperplane that does not intersect $\Sigma_{K_x}$. 
	Now let $e_1,e_2,\ldots,e_k$ be the edges of $\rho$ with corresponding dual hyperplanes $W_1,W_2,\ldots, W_k$. If $W_{i+1}$ intersects $\Sigma_{K_x}$, then there is a disk diagram $D$ enclosed by the path $e_1\ldots e_ie_{i+1}$, a curve in the carrier of $W_{i+1}$ and a path in $K_x$. Since $W_i$ does not intersect $\Sigma_{K_x}$, it must exit $D$ by crossing $W_{i+1}$. 
	
By \cite[Lemma 3.6]{WiseBook} $W_{i+1}$ and $W_i$ cannot interosculate. Therefore, $e_i$ and $e_{i+1}$ must corner a square. Let $e_i'$ and $e_{i+1}'$ be the edges opposite $e_i$ and $e_{i+1}$ respectively.
Then let
\[\sigma' = e_1e_2\ldots e_{i-1}e_{i+1}'e_i' e_{i+2}\ldots e_k.\] 
Now $W_{i+1}$ is the hyperplane dual to the $i$th edge of $\sigma'$ from $v$, which violates the inductive hypothesis. 
Hence $W_{i+1}$ cannot intersect $\Sigma_{K_x}$. 
\end{proof}

	Before continuing with the proof of \Cref{T: hyperplane separation theorem} we require one additional auxiliary fact.  It is clear that geodesics in $\Gamma$ lift to quasi-geodesics in $\tilde{X}$ via the de-electrification map in \Cref{P: electric qi}.  We need a way to show that long enough quasi-geodesics in $\tilde{X}$ escape any finite neighborhood of some hyperplane.
	Recall the \textbf{Ramsey number} $Ram(a,b)$ is the number of vertices such that any graph on $Ram(a,b)$ vertices either contains a complete graph of size $a$ or its complement contains a complete graph of size $b$ (see \cite{RamseyTheory} for more about Ramsey numbers).  We may associate to any CAT$(0)$ cube complex its \textbf{crossing graph} with vertex set corresponding to hyperplanes and two vertices are adjacent if and only if their associated hyperplanes cross (see for example \cite{Hagen} for more details on crossing and related graphs).

	\begin{lemma}
		\label{Ramsey distance}
		Let $\Omega$ be an arbirarty CAT$(0)$ cube complex with $d = \dim(\Omega)$.
		Let $\alpha$ be an arbitrary combinatorial geodesic in $\Omega$.
		For any positive integer $N > 0$, any collection $\cal{H}$ consisting of at least $R = Ram(d+1, N)$ distinct hyperplanes all dual to edges of $\alpha$ contains a subset $\{ \hyp_1, \dotsc, \hyp_{N} \} \subseteq \cal{H}$ that form a nested sequence of $N$--halfspaces $\hyp_1^+ \subsetneq \hyp_2^+ \subsetneq \cdots \subsetneq  \hyp_{N}^+$.
		
		In particular, if the geodesic $\alpha$ has length at least $N$, then there exists a hyperplane $\hyp$ dual to an edge of $\alpha$ and a vertex $\alpha(t)$ such that $d(\alpha(t), \hyp) \geq N - 1$.  
	\end{lemma}
	\begin{proof}
		Let $\cal{C}$ denote the crossing graph of $\Omega$.  
		Since $\tilde{X}$ has finite dimension, any collection of pairwise crossing hyperplanes in $\Omega$ has cardinality at most $d$, so $\mc{C}$ does not contain a complete graph on $d+1$ vertices.  
		Hyperplanes dual to edges of a combinatorial geodesic are distinct by work of Sageev \cite[Theorem~4.13]{Sageev95}, so the hyperplanes dual to edges of $\alpha$ correspond to an induced sub-graph of $\cal{C}$.  
		
		Hence, any collection of at least $R$ hyperplanes dual to $\alpha$ contains a subset $\{\hyp_1, \dotsc, \hyp_N\}$ that pairwise do not cross.  Each $\hyp_i$ is dual to an edge of $\alpha$, so can be totally ordered by picking an orientation on $\alpha$.  This orientation corresponds to a choice of halfspace $\hyp_i^+$, which are clearly nested.
		
		To see the last statement of \Cref{Ramsey distance}, observe that, if $\hyp_1^+ \subsetneq \hyp_2^+ \subsetneq \cdots \subsetneq  \hyp_N^+$ is a sequence of $N$ nested halfspaces, then any points $p \in \hyp_1$ and $q \in \hyp_N$ are distance $d(p,q) \geq N-1$ apart.  Thus, for $p = \alpha \cap \hyp_1$ we have $d(p, \hyp_N) \geq N-1$.
	\end{proof}

	As we saw in the proof of \Cref{Ramsey distance}, (combinatorial) geodesics in CAT$(0)$ cube complexes may only cross a given hyperplane at most once.  On the other hand, (infinite) quasi-geodesics may cross a given hyperplane (infinitely) many times even when the underlying complex is locally finite.  Relatively geometric actions on CAT$(0)$ cube complexes give up local finiteness, but requires that the cube complex also be $\delta$--hyperbolic.  In \Cref{biinfinite sidedness,ray sidedness} we will show that in $\delta$--hyperbolic CAT$(0)$ cube complexes there is a choice of hyperplane whose interactions with a given quasi-geodesic has many of the same useful properties of a hyperplane dual to an honest geodesic.    
	A subspace $Y$ of a geodesic metric space $X$ is call \textbf{Morse} when for every $A > 0$ and $B \geq 0$ there exists a contant $D = D(A,B) \geq 0$ such that any $(A,B)$--quasi-geodesic joining points in $Y$ is contained in the $D$-neighborhood of $Y$.  
	We call $D$ the \textbf{Morse constant} for the quasi-geodesic parameters $(A,B)$.

	\begin{lemma}\label{biinfinite sidedness}
	Let $\tilde{X}$ be a $\delta$--hyperbolic CAT(0) cube complex. Let $\gamma:(-\infty,\infty)\to \tilde{X}\oskel$ be a connected bi-infinite combinatorial $(\lambda,\epsilon)$--quasi-geodesic. Given $M>1$, there exist a hyperplane $W$ of $\tilde{X}$ and $t_M>0$ so that for all $t$ with $|t|>t_M$:
	\begin{enumerate}
	\item $\gamma(\pm t)$ lie in distinct complementary components of $W$,
	\item $\gamma$ crosses $W$ an odd number of times, and
	\item $d(\gamma(t),W)>M$.
	\end{enumerate}
	In particular, $\gamma(t)\in W$ implies $|t|\le t_M$. 
	\end{lemma}
	
	\begin{proof}
		For any $t_1, t_2 \in \mathbb{R}$, we write $[\gamma(t_1),\gamma(t_2)]$ to mean any combinatorial geodesic connecting $\gamma(t_1)$ to $\gamma(t_2)$.  Note also that the $M$--neighborhood of a convex subset of $\tilde{X}$ is necessarily Morse.  
		
		Let $D = D(\lambda, \varepsilon) > 0$ be the Morse constant for the quasi-geodesic parameters of $\gamma$.
		Let $R$ be the constant from \Cref{Ramsey distance} with $N = 2(D + M + 1) + 1$.  
		Choose
		\[
		t_M >  \lambda(R + \varepsilon).
		\]
		Let $\gamma_0 = [\gamma(-t_M), \gamma(t_M)]$. 
		By choice of $t_M$, we have $\diam(\gamma_0) > R$.  
		\Cref{Ramsey distance} guarantees that there exist hyperplanes $\{\hyp_i :  -(D + M + 1) \leq i \leq D+M+1\}$ all dual to edges of $\gamma_0$ that form the following nested sequence of halfspaces: 
		\[
		\hyp_{-(D +M +  1)} \subsetneq \cdots \subsetneq \hyp_{-1} \subsetneq \hyp_0 \subsetneq \hyp_1 \subsetneq \cdots \subsetneq \hyp_{D + M + 1}.
		\]
		We will see that we may choose $W = \hyp_0$.
		Since the halfspaces of the $\hyp_i$ are nested:
		\begin{equation}\label{eq: dist}
		\min \{ d(\gamma(-t_M), \hyp_0), \, d(\gamma(t_M), \hyp_0) \} > D + M. \tag{$\dagger\dagger$}
		\end{equation}
		The concatenation,
		\[
		\Upsilon : = \gamma\vert_{(-\infty, -t_M)} \cup \gamma_0 \cup \gamma\vert_{(t_M, \infty)},
		\]
		is again a $(\lambda, \varepsilon)$--quasi-geodesic. Since $\gamma_0$ crosses $\hyp_0$, if $\gamma(t)\in\mc{N}_M(\hyp_0)$ for some $t$ with $t>t_M$ then $\gamma(t_M)\in \mc{N}_{D+M}(\hyp_0)$, contrary to \eqref{eq: dist}. We conclude that if $t>t_M$, $\gamma(t_M)\notin \mc{N}_M(\hyp_0)$. Similarly, if $t<-t_M$, $\gamma(t)\notin \mc{N}_M(\hyp_0)$.   
		Since $\Upsilon$ and $\gamma$ coincide for all $|t| > t_M$, we immediately have that $\gamma(t)$ and $\gamma(-t)$ lie in distinct components of $\tilde{X}\, \,\setminus\, \mc{N}_{M}(\hyp_0)$.  Since $\gamma_0$ crosses $\hyp_0$ once, $\gamma$ crosses $\hyp_0$ an odd number of times.  
	\end{proof}

	It remains to account for the situation where $\gamma$ joins a parabolic point to a conical limit point.
	Using a similar argument to \Cref{biinfinite sidedness}, it is possible to prove:
	
	\begin{lemma}\label{ray sidedness}
	Let $\gamma:[0,\infty)\to \tilde{X}$ be an infinite combinatorial $(\lambda,\epsilon)$--quasi-geodesic ray in a $\delta$--hyperbolic CAT(0) cube complex $\tilde{X}$ where $x = \gamma(0)$ is a vertex. Given $M>1$, there exists a hyperplane $W$ of $\tilde{X}$ and $t_M>0$ so that for all $t$ with $t>t_M$:
		\begin{enumerate}
		\item $W$ separates $\gamma(0)$ and $\gamma(t)$,
		\item $\Sigma_{\Stab_K(x)}\cap W = \emptyset$,
		\item $d(\gamma(t),W)>M$. 
		\end{enumerate}
		In particular, if $\gamma(t)\in W$ then $0<t\le t_M$. 
	\end{lemma}
	
	The same strategy used in Lemma~\ref{biinfinite sidedness} works to prove \Cref{ray sidedness} with the following adjustments:
	\begin{itemize}
	\item the hyperplane $W$ should be the hyperplane dual to an edge in the middle of a geodesic joining $\gamma(0)$ and $\gamma(t)$ for some suitable $t>>0$. 
	\item When $t>> \max_{D\in\mc{D}}\{\diam \Sigma_D\}$, then $\Sigma_{\Stab_K(x)}\cap W = \emptyset$.
	\end{itemize}
	
	

We summarize the discussion above in the context of separating points in the Bowditch boundary as follows:
	
	\begin{proposition}\label{P: hyperplane separation}
	Let $x,y \in \partial_{\mc{D}}K$ be distinct. There exists a hyperplane that separates $x,y$ in the sense of Definition~\ref{D: upstairs separation} with respect to the action of $K$ on $\tilde{X}$. 
	\end{proposition}
	
	\begin{proof}
	If $x,y$ are both conical limit points, let $\gamma$ be a bi-infinite geodesic between $x$ and $y$. Then the complete de-electrification $\hat\gamma$ of $\gamma$ is a connected bi-infinite quasi-geodesic such that $\beta(\hat\gamma)=\gamma$.  
	Lemma~\ref{biinfinite sidedness} allows us to choose a hyperplane in $\tilde{X}$ so that if $|t|>>0$, $\hat\gamma(t)$ does not cross $W$, $\hat\gamma(\pm t)$ are on opposite sides of $W$ and $d(\hat\gamma(t),W)>\max_{D\in \mc{D}} \{\diam \Sigma_{D}+2\delta\}$. 
	Therefore for $|t|>>0$, $\gamma(\pm t)$ each lie in distinct components of $\Gamma\,\setminus\, \beta(L)$ by Lemma~\ref{biinfinite sidedness}.
	For all $S>0$, $\diam\{t:\,d(\gamma(t),W)\le S\}<\infty$ because otherwise a standard hyperbolic geometry argument using a thin quadrilateral shows that $\diam\{t:\,d(\gamma(t),W)<2\delta\}$ is not bounded above, which contradicts our choice of $W$. Therefore, $\lim_{t\to\pm\infty} d(\gamma(t),W) =\infty$, and $x,y$ are not in the limit set of the stabilizer of $W$. 
	
	If one of $x,y$ is a conical limit point, assume without loss of generality that $x$ is the conical limit point and $y$ is a parabolic point. 
	Apply the argument from the previous case except use Lemma~\ref{ray sidedness} in place of Lemma~\ref{biinfinite sidedness} to extract the desired hyperplane $W$. 
	To see that $y\notin \Lambda \Stab_K(W)$, observe that $W\cap \Sigma_{\Stab_K(x)} =\emptyset$ and apply Proposition~\ref{P: hyperplane parabolic iff subgraph intersects}. 
	
	If $x,y$ are both distinct parabolic vertices in $\Gamma$, let $\hat\gamma$ be a minimal length geodesic in $\tilde{X}\oskel$ between $\Sigma_{\Stab_K(x)}$ and $\Sigma_{\Stab_K(y)}$. 
	Then $\gamma$ has an edge with finite stabilizer because $\Stab_K(x)\cap \Stab_K(y)$ is finite.
	Let $W$ be the hyperplane dual to this edge. 
	We see that $x,y\notin\Lambda \Stab_K(W)$ by Lemma~\ref{L: stabilizer classification} and Proposition~\ref{P: hyperplane parabolic iff subgraph intersects}.
	\end{proof}

	Finally, we prove Theorem~\ref{T: hyperplane separation theorem} from the introduction.
	
	\relgeomseparation*
	
	
\begin{proof}
Let $W$ be the separating hyperplane specified by Proposition~\ref{P: hyperplane separation} with associated hyperset $L$. 
Then $\beta(L)$ separates $x,y$ and its carrier $\beta(J)$ are quasi-convex. 
Since $K$ acts relatively geometrically, $K$ acts cocompactly on $\tilde{X}$ and on $\Gamma$, see Proposition~\ref{P: electric qi}.  In particular, $\tilde{X}$ is finite dimensional. Edge stabilizers are finite because each maximal parabolic stabilizes exactly one vertex and all cell stabilizers are parabolic. 
By Proposition~\ref{P: hyperplane two-sided carrier}, $\beta(J)$ has the two-sided carrier property. 
Therefore, by Theorem~\ref{T: separation criterion}, there exists a subgroup ${K_W}$ of index at most 2 in $\Stab_G(W)$ so that $x,y$ are in ${K_W}$--distinct components of $\partial_{\mc{D}}K \,\setminus\, \Lambda {K_W}$.  
	\end{proof}

	\bibliography{../../cubes}
	\bibliographystyle{alpha}
\end{document}